\documentclass[11pt]{amsart}
\usepackage{amsmath}
\usepackage{amsthm}
\usepackage{amssymb}
\usepackage{verbatim}
\usepackage{color}

\newtheorem{theorem}{Theorem}    

\newtheorem{corollary}[theorem]{Corollary}
\newtheorem{lemma}[theorem]{Lemma}

\newtheorem{remark}[theorem]{Remark}
\newtheorem{definition}[theorem]{Definition}
\theoremstyle{definition}

\numberwithin{theorem}{section} \numberwithin{theorem}{section}
\numberwithin{equation}{section}

\allowdisplaybreaks

\begin{document}
\title[Compactness of commutators in the Bloom setting: the off-diagonal case]
{Compactness of commutators in the Bloom setting: the off-diagonal case}

\author{Yongming Wen$^*$}


%
\thanks{$^*$Corresponding author.}
\thanks{Supported by the Natural Science Foundation of Fujian Province(Nos. 2021J05188), the scientific research project of The Education Department of Fujian Province(No.JAT200331), President’s fund of Minnan Normal University (No. KJ2020020), Institute of Meteorological Big Data-Digital Fujian, Fujian Key Laboratory of Data Science and Statistics and Fujian Key Laboratory of Granular Computing and Applications (Minnan Normal University), China.}
\address{School of Mathematics and Statistics, Minnan Normal University, Zhangzhou 363000,  China} \email{wenyongmingxmu@163.com}



\begin{abstract}
In this paper, we establish a new characterization of weighted $\rm{VMO}$ spaces, which are essential different from the classical $\rm{VMO}$ spaces, via the compactness of sparse operators, commutators of  Riesz potentials and fractional maximal operators. Some of our results are new even in unweighted setting.
\end{abstract}

\maketitle

\section{Introduction and main results}
It is well known that if a function $b\in\rm{BMO}(\mathbb{R}^n)$, then the commutator $[b,T]f:=bT(f)-T(bf)$ of a Calder\'{o}n-Zygmund operator is bounded on $L^p(\mathbb{R}^n)$ for any $1<p<\infty$, see \cite{CRW}. Particularly, Coifman et al. \cite{CRW} proved that the commutator of Riesz transform is bounded on $L^p(\mathbb{R}^n)$ if and only if $b\in\rm{BMO}(\mathbb{R}^n)$. Subsequently, Janson \cite{Jan} and Uchiyama \cite{U} independently established the full converse of the work in \cite{CRW}. Moreover, Uchiyama \cite{U} showed that the $L^p$-boundedness of $[b,T]$ could be refined to a compactness one if the function space $\rm{BMO}(\mathbb{R}^n)$ is replaced by $\rm{VMO}(\mathbb{R}^n)$, which is defined to be the closure of $C_c^\infty(\mathbb{R}^n)$ in the $\rm{BMO}(\mathbb{R}^n)$ topology. Since then, the study of compactness of commutators has blossomed, see \cite{BeDMT,ChaT,ChD,ChDW,ClC,GHWY,GWY,GZ,ShFL,TX,WaX,Wa,WY,XYY} and the references there in. Among them, Clop and Cruz \cite{ClC} first established the weighted compactness of commutators of Calder\'{o}n-Zygmund operators. Wu and Yang \cite{WY} established the characterizations of weighted compactness of Riesz transform and Riesz potential via $\rm{VMO}(\mathbb{R}^n)$.

Recently, equivalent characterizations of two-weight norm inequalities for commutators of singular integrals and related operators have attracted lots of attentions. This topic is originated from Bloom \cite{Bl}, who extended the work of Coifman et al. \cite{CRW} to the following two weight version:
$$\|[b,H]\|_{L^p(\lambda_1)\to L^p(\lambda_2)}\simeq \|b\|_{ BMO_\nu(\mathbb R)}\, \quad \text{for}\,\, 1<p<\infty,\,\nu=(\lambda_1/\lambda_2)^{1/p},\,\lambda_1,\,\lambda_2\in A_p,$$
where $H$ is the Hilbert transform and $\rm{BMO}_\nu(\mathbb R^n)$ is the following weighted $\rm{BMO}$ space defined by
$${\rm{BMO}}_\nu(\mathbb R^n):=\Big\{b\in L^1_{\rm loc}(\mathbb{R}^n):\,\, \|b\|_{\rm{BMO}_\nu(\mathbb R^n)}:=\sup_{B\subset\mathbb{R}^n}\frac1{\nu(B)}\int_B|b(y)-\langle b\rangle_B|dy<\infty\Big\},$$
where $\langle b\rangle_B=|B|^{-1}\int_Bb$. It is easy to see that $\rm{BMO}_\nu(\mathbb{R}^n)$ is just the classical $\rm{BMO}(\mathbb{R}^n)$ space when $\nu=1$. In 2016, Holmes et al. \cite{HLW,HRS} extended Bloom's result to commutators of Riesz transforms and fractional integrals. Furthermore, Lerner, Ombrosi and Rivera-R\'{i}os \cite{LOR,LOR1} obtained quantitative weighted estimates for the commutators of Calderon-Zygmund operators $[b,T]$ via the following pointwise sparse dominations:
$$|[b,T]f(x)|\lesssim\sum_{j=1}^{3^n}\Big(\mathcal{T}_{\mathcal{S}_j,b}(|f|)(x)+
(\mathcal{T}_{\mathcal{S}_j,b})^\ast(|f|)(x)\Big),$$
where $\mathcal{S}_j$ is a sparse family,
\begin{equation}\label{sparse1}
\mathcal{T}_{\mathcal{S},b}(f)(x)=\sum_{Q\in\mathcal{S}}|b(x)-\langle b\rangle_Q|\langle f\rangle_Q\chi_Q(x)
\end{equation}
and
\begin{align*}
(\mathcal{T}_{\mathcal{S},b})^\ast(f)(x)=\sum_{Q\in\mathcal{S}}\langle|b-\langle b\rangle_Q|f\rangle_Q\chi_Q(x).
\end{align*}
Besides, Lerner et al. \cite{LOR1} also gave a characterization of the two-weighted boundedness of commutators of homogeneous singular integrals. Following the work of \cite{LOR1}, Accomazzo et al. \cite{AMPRR} improved the work in \cite{HRS} by proving that
$$[b,I_\alpha]f(x)=\int_{\mathbb{R}^n}(b(x)-b(y))\frac{f(y)}{|x-y|^{n-\alpha}}dy$$
is dominated by the sparse operators
$$\sum_{j=1}^{3^n}\Big(\mathcal{T}_{S_j,b}^\alpha(|f|)(x)+
(\mathcal{T}_{S_j,b}^\alpha)^\ast(|f|)(x)\Big),$$
where
\begin{equation}\label{sparse2}
\mathcal{T}_{\mathcal{S},b}^\alpha(f)(x)=\sum_{Q\in\mathcal{S}}|Q|^{\alpha/n}|b(x)-\langle b\rangle_Q|\langle
f\rangle_Q\chi_Q(x)
\end{equation}
and
\begin{equation}\label{sparse3}
(\mathcal{T}_{\mathcal{S},b}^\alpha)^\ast(f)(x)=\sum_{Q\in\mathcal{S}}|Q|^{\alpha/n}\langle
|b-\langle b\rangle_Q|f\rangle_Q\chi_Q(x).
\end{equation}
And they as well characterized the two weight estimates of $[b,I_\alpha]$, i.e., for $1<p<n/\alpha$, $1/p-1/q=\alpha/n$, $\lambda_1,\,\lambda_2\in A_{p,q}$, $\nu=\lambda_1/\lambda_2$,
$$\|[b,I_\alpha]\|_{L^p(\lambda_1^p)\to L^q(\lambda_2^q)}\simeq \|b\|_{ BMO_\nu(\mathbb{R}^n)}.$$

However, the study of the compactness of commutators in the two weight setting has just begun by the recent work of Lacey and Li \cite{LL}, in which the authors showed that $[b,R_j]$ is compact from $L^{p}(\lambda_1)$ to $L^{p}(\lambda_2)$ if and only if $b\in \rm{VMO}_\nu(\mathbb{R}^n)$, where $1<p<\infty$, $\lambda_1,\lambda_2\in A_p$, $\nu=(\lambda_1/\lambda_2)^{1/p}$. Moreover, they pointed out that $\rm{VMO}_\nu(\mathbb{R}^n)$ spaces are essential different from the classical $\rm{VMO}(\mathbb{R}^n)$ spaces in the sense that $C_c^\infty(\mathbb{R}^n)$ need not be contained in $\rm{BMO}_\nu(\mathbb{R}^n)$ for $n\geq2$. Later on, Liu, Wu and Yang \cite{LiWY} established the two-weight compactness for commutators of Calder\'{o}n-Zygmund operators and fractional integral operators via the extrapolation theorems of compactness. Recently, Chen et al. \cite{ChLLV} further established the characterization of compactness for the sparse operator $\mathcal{T}_{\mathcal{S},b}$ defined as \eqref{sparse1} in the two weight setting on the homogeneous spaces, as applications, they obtained the compactness characterization for the commutators of Hardy-Littlewood maximal operators and commutators of Calder\'{o}n-Zygmund operators with kernels satisfying the non-degenerate conditions. An interesting and natural question would ask: How is the case for sparse operator $\mathcal{T}_{\mathcal{S},b}^\alpha$ defined as \eqref{sparse2}? Our first result is to address this question.

\begin{theorem}\label{theorem1.1}
Let $\alpha\in(0,n)$, $1<p<n/\alpha$, $1/p-1/q=\alpha/n$. Assume that $\lambda_1,\lambda_2\in A_{p,q}$, $\nu=\lambda_1/\lambda_2$, $b\in \rm{BMO}_\nu(\mathbb{R}^n)$ and $\mathcal{T}_{\mathcal{S},b}^\alpha$ is defined as \eqref{sparse2}.

\medskip

{\rm (1)}\, If $b\in\rm{VMO}_{\nu}(\mathbb{R}^n)$, then for any $\eta\in(0,1)$ and each $\eta$-sparse family $\mathcal{S}$, $\mathcal{T}_{\mathcal{S},\alpha}^b$ is compact from $L^p(\lambda_1^p)$ to $L^q(\lambda_2^q)$;

\medskip

{\rm (2)}\, If for any $\eta\in(0,1)$ and each $\eta$-sparse family $\mathcal{S}$, $\mathcal{T}_{\mathcal{S},\alpha}^b$ is compact from $L^p(\lambda_1^p)$ to $L^q(\lambda_2^q)$, then $b\in\rm{VMO}_{\nu}(\mathbb{R}^n)$.
\end{theorem}

Next, we consider the commutator of fractional maximal operator $M_\alpha^b$, which is defined by
\begin{align*}
M_\alpha^b(f)(x)=\sup_{B\ni x}\frac{1}{|B|^{1-\alpha/n}}\int_B|b(x)-b(y)||f(y)|dy.
\end{align*}
Segovia and Torrea \cite{ST} proved that
\begin{align*}
\|M_\alpha^b\|_{L^p(\lambda_1^p)\rightarrow L^q(\lambda_2^q)}\simeq\|b\|_{\rm{BMO}_\nu},
\end{align*}
where $\alpha\in(0,n)$, $1<p<n/\alpha$, $1/p-1/q=\alpha/n$, $\lambda_1,\lambda_2\in A_{p,q}$ and $\nu=\lambda_1/\lambda_2$. Our second result extends this result to the compactness one.
\begin{theorem}\label{theorem1.2}
Let $\alpha\in(0,n)$, $1<p<n/\alpha$, $1/p-1/q=\alpha/n$. Assume that $\lambda_1,\lambda_2\in A_{p,q}$, $\nu=\lambda_1/\lambda_2$ and $b\in L_{\rm{loc}}^{1}(\mathbb{R}^n)$.
\begin{itemize}
\item [(1)]If $b\in \rm{VMO}_\nu(\mathbb{R}^n)$, then $M_\alpha^b$ is compact from $L^p(\lambda_1^p)$ to $L^q(\lambda_2^q)$;
\item [(2)]If $M_\alpha^b$ is compact from $L^p(\lambda_1^p)$ to $L^q(\lambda_2^q)$, then $b\in \rm{VMO}_\nu(\mathbb{R}^n)$.
  \end{itemize}
\end{theorem}

The other definition of commutator of fractional maximal operator is given as follows:
\begin{align*}
[b,M_\alpha](f)(x)=b(x)M_\alpha f(x)-M_\alpha(bf)(x),~\alpha\in(0,n),
\end{align*}
where $M_\alpha$ is the fractional maximal operator defined by
$$M_\alpha f(x)=\sup_{B\ni x}\frac{1}{|B|^{1-\alpha/n}}\int_B|f(y)|dy.$$
$M_\alpha^b$ and $[b,M_\alpha]$ essentially differ from each other. For instance, $M_\alpha^b$ is a sublinear and positive operator, but it is not for $[b,M_\alpha]$.
As a corollary of Theorem \ref{theorem1.1}, we have the following two-weight compactness for $[b,M_\alpha]$.
\begin{corollary}\label{corollary1.3}
Let $\alpha\in(0,n)$, $1<p<n/\alpha$, $1/p-1/q=\alpha/n$ and $b$ be a non-negative locally integrable function. Assume that $\lambda_1,\lambda_2\in A_{p,q}$, $\nu=\lambda_1/\lambda_2$ and $b\in \rm{VMO}_\nu(\mathbb{R}^n)$, Then $[b,M_\alpha]$ is compact from $L^p(\lambda_1^p)$ to $L^q(\lambda_2^q)$.
\end{corollary}
\begin{remark}
Theorem \ref{theorem1.1} and Corollary \ref{corollary1.3} are new even in unweighted setting.
\end{remark}

As we stated before, Liu, Wu and Yang \cite{LiWY} gave the following two weight compactness of commutator of fractional integral:
\begin{theorem}{\rm(cf. \cite{LiWY})}
Let $0<\alpha<n$, $1<p<n/\alpha$, $1/p-1/q=\alpha/n$ and $\lambda_1,\lambda_2\in A_{p,q}$. Suppose that
\begin{equation}\label{condition}
b\in
\overline{\bigcup_{\eta\in(1,r_\nu]}(\rm{BMO}_{\nu^\eta}(\mathbb{R}^n))
\cap\rm{CMO}(\mathbb{R}^n)}^{\rm{BMO}_\nu(\mathbb{R}^n)},
\end{equation}
where
$$\nu=\lambda_1/\lambda_2, r_\nu:=1+\frac{1}{2^{5+n}[\nu]_{A_2}}.$$
Then the commutator $[b,I_\alpha]$ is compact from $L^p(\lambda_1^p)$ to $L^q(\lambda_2^q)$.
\end{theorem}
Liu et al. \cite{LiWY} only proved that the condition \eqref{condition} is equivalent to $b\in\rm{VMO}_\nu(\mathbb{R}^n)$ when $n=1$. Hence, it is natural to consider that whether $b\in\rm{VMO}_\nu(\mathbb{R}^n)$ is a sufficient condition for the two weight compactness of $[b,I_\alpha]$, and the converse is true or not? Our next theorem is to settle this question.

\begin{theorem}\label{theorem1.6}
Let $\alpha\in(0,n)$, $1<p<n/\alpha$, $1/p-1/q=\alpha/n$. Assume that $\lambda_1,\lambda_2\in A_{p,q}$, $\nu=\lambda_1/\lambda_2$ and $b\in L_{\rm{loc}}^{1}(\mathbb{R}^n)$.
\begin{itemize}
\item [(1)]If $b\in \rm{VMO}_\nu(\mathbb{R}^n)$, then $[b,I_\alpha]$ is compact from $L^p(\lambda_1^p)$ to $L^q(\lambda_2^q)$;
\item [(2)]If $[b,I_\alpha]$ is compact from $L^p(\lambda_1^p)$ to $L^q(\lambda_2^q)$, then $b\in \rm{VMO}_\nu(\mathbb{R}^n)$.
  \end{itemize}
\end{theorem}
By Theorems \ref{theorem1.1}, \ref{theorem1.2} and \ref{theorem1.6}, we have the following corollary.
\begin{corollary}\label{corollary1.7}
Let $\alpha\in(0,n)$, $1<p<n/\alpha$, $1/p-1/q=\alpha/n$. Assume that $\lambda_1,\lambda_2\in A_{p,q}$, $\nu=\lambda_1/\lambda_2$ and $b\in \rm{BMO}_\nu(\mathbb{R}^n)$. Then the following statements are equivalent:
\begin{itemize}
\item [(1)]For any $\eta\in(0,1)$ and each $\eta$-sparse family $\mathcal{S}$, $\mathcal{T}_{\mathcal{S},\alpha}^b$ is compact from $L^p(\lambda_1^p)$ to $L^q(\lambda_2^q)$;
\item [(2)]$[b,I_\alpha]$ is compact from $L^p(\lambda_1^p)$ to $L^q(\lambda_2^q)$;
\item [(3)]$M_\alpha^b$ is compact from $L^p(\lambda_1^p)$ to $L^q(\lambda_2^q)$;
\item [(4)]$b\in\rm{VMO}_\nu(\mathbb{R}^n)$.
  \end{itemize}
\end{corollary}

\begin{remark}\label{r-heat and poisn}
To the best of our knowledge, the necessity of the two weight compactness of commutator of Riesz potential and the characterization of two weight compactness of commutator of fractional maximal operator are new results.
\end{remark}

The rest of the paper is organized as follows. In Section 2, we recall some necessary facts that will be used. The proofs of $(1)$ of Theorems \ref{theorem1.1}, \ref{theorem1.2}, \ref{theorem1.6} and Corollary \ref{corollary1.3} will be given in Sections 3. In Section 4, we give the proofs of (2) in Theorems \ref{theorem1.1}, \ref{theorem1.2}, \ref{theorem1.6}. We remark that our works and ideas are motivated by \cite{ChLLV,LL} but with some differences due to the off-diagonal case we consider.

We end this section by making some conventions. Denote $f\lesssim g$, $f\thicksim g$ if $f\leq Cg$ and $f\lesssim g \lesssim f$, respectively.  For any ball $B:=B(x_0,r_B)\subset \mathbb{R}^n$, $x_0$ and $r_B$ are the center and the radius of the ball $B$, respectively. $\langle f\rangle_B$ means the mean value of $f$ over $B$, $\chi_B$ represents the characteristic function of $B$, $\int_B\omega(y)dy$ is denoted by $\omega(B)$.
\section{Preliminaries}
In this section, we recall some basic facts that will be used.
\subsection{Weight}
A weight $\omega$ is a non-negative locally integrable function on $\mathbb{R}^n$. For $1<p<q<\infty$, we say that $\omega\in A_{p,q}$ if
\begin{align*}
[\omega]_{A_{p,q}}:=\sup_Q\Big(\frac{1}{|Q|}\int_Q\omega(y)^qdy\Big)
\Big(\frac{1}{|Q|}\int_Q\omega(y)^{-p'}dy\Big)^{q/p'}<\infty,
\end{align*}
where $1/q+1/q'=1$ and the supremum is taken over all cubes $Q\subset \mathbb{R}^n$. It is easy to see that $\omega^p\in A_p$ and $\omega^q\in A_q$, where $A_p(1<p<\infty)$ is the classical Muckenhoupt class, i.e., we say that $\omega\in A_p$ if
\begin{align*}
[\omega]_{A_p}:=\sup_{Q}\Big(\frac{1}{|Q|}\int_{Q}\omega(y)dy\Big)\Big(\frac{1}{|Q|}
\int_{Q}\omega(y)^{1-p'}dy\Big)^{p-1}<\infty,
\end{align*}
where $1/p+1/p'=1$ and the supremum is taken over all cubes $Q\subset \mathbb{R}^n$. It is well known that if $\omega\in A_p$, there are constants $C_1,C_2>0$ and $0<\sigma<1$ such that
\begin{align}\label{Ap property}
C_1\Big(\frac{|E|}{|B|}\Big)^p\leq\frac{\omega(E)}{\omega(B)}\leq C_2\Big(\frac{|E|}{|B|}\Big)^\sigma,
\end{align}
for any measurable subset $E$ of a ball $B$.
\subsection{Sparse operators}
Now we recall the definitions of dyadic lattice, sparse family and sparse operator; see, for example,  \cite{Ler,LOR}. Given a cube $Q\subset\mathbb{R}^n$, let $\mathcal{D}(Q)$ be the collection of cubes obtained by repeatedly subdividing $Q$ and its descendants into $2^n$ congruent subcubes.
\begin{definition}
A collection of cubes $\mathcal{D}$ is called a dyadic lattice if it satisfies the following properties:\\
$(1)$ if $Q\in\mathcal{D}$, then every child of $Q$ is also in $\mathcal{D}$;\\
$(2)$ for every two cubes $Q_1, Q_2\in\mathcal{D}$, there is a common ancestor $Q\in\mathcal{D}$ such that $Q_1, Q_2\in\mathcal{D}(Q)$;\\
$(3)$ for any compact set $K\subset\mathbb{R}^n$, there is a cube $Q\in\mathcal{D}$ such that $K\subset Q$.
\end{definition}

\begin{definition}
Given $\eta\in(0,1)$, a subset $\mathcal{S}\subset\mathcal{D}$ is called an $\eta$-sparse family provided that for each cube $Q\in\mathcal{S}$, there is a measurable subset $E_Q\subset Q$ such that $\eta|Q|\leq|E_Q|$, and the sets $\{E_Q\}_{Q\in\mathcal{S}}$ are mutually disjoint.
\end{definition}

Let $\mathcal S$ be a sparse family. Define the sparse operators $\mathcal{T}_{\mathcal{S}}$ and $\mathcal{T}_{\mathcal{S}}^\alpha$ by
\begin{align}\label{2.2}
\mathcal{T}_{\mathcal S}f(x):=\sum_{Q\in \mathcal S}\langle|f|\rangle_Q\chi_Q(x)
\end{align}
and
\begin{align}\label{2.3}
\mathcal{T}_{\mathcal{S}}^\alpha f(x):=\sum_{Q\in\mathcal{S}}\frac{1}{|Q|^{1-\alpha/n}}\int_Q|f|\chi_Q(x),
\end{align}
respectively. Then
\begin{equation}\label{sparse boundedness1}
\|\mathcal{T}_{\mathcal{S}}f\|_{L^p(\omega)}\lesssim[\omega]_{A_p}^
{\max\{1,\frac{1}{p-1}\}}\|f\|_{L^p(\omega)},~\omega\in A_p,~1<p<\infty,
\end{equation}
see \cite{Ler}. For $\alpha\in(0,n)$, $1<p<n/\alpha$, $1/p-1/q=\alpha/n$ and $\omega\in A_{p,q}$,
\begin{align}\label{sparse boundedness2}
\|\mathcal{T}_\mathcal{S}^\alpha\|_{L^p(\omega^p)\rightarrow L^q(\omega^q)}\lesssim[\omega]_{A_{p,q}}^{(1-\alpha/n)\max\{1,p/q'\}},
\end{align}
the proof of \eqref{sparse boundedness2} is implicit in \cite{LMPT}.

We recall a lemma that will be used.
\begin{lemma}{\rm (cf. \cite{LOR})}\label{lemma2.3}
Let $\mathcal{D}$ be a dyadic lattice and let $\mathcal{S}\subset\mathcal{D}$ be a $\eta$-sparse family. Assume that $b\in L_{\rm{loc}}^{1}(\mathbb{R}^n)$. Then there exists a $\frac{\eta}{2(1+\eta)}$-sparse family $\tilde{\mathcal{S}}\subset\mathcal{D}$ such that $\mathcal{S}\subset\tilde{\mathcal{S}}$ and for every cube $Q\in\tilde{\mathcal{S}}$,
$$|b(x)-\langle b\rangle_Q|\leq2^{n+2}\sum_{R\in\tilde{\mathcal{S}},R\subset Q}\frac{1}{|R|}\int_R|b(x)-\langle b\rangle_R|dx\chi_R(x)$$
for a.e. $x\in Q$.
\end{lemma}

\subsection{Weighted $\rm{VMO}$ spaces}
The weighted $\rm{VMO}$ spaces were recently introduced by Lacey and Li \cite{LL} to character the two-weight compactness of commutators of Riesz transforms, see its definition below:
\begin{definition}\label{weighted VMO}
Let $1<p<\infty$ and $\lambda_1,\lambda_2\in A_p$, $\nu=\lambda_1^{1/p}\lambda_2^{-1/p}$. A function $b\in\rm{BMO}_\nu(\mathbb{R}^n)$ belongs to $\rm{VMO}_\nu(\mathbb{R}^n)$ if
\begin{itemize}
\item [(1)] $\lim\limits_{a\rightarrow0}\sup\limits_{B\subset\mathbb{R}^n,r_B=a}\frac{1}{\nu(B)}\int_{B}
|b(x)-\langle b\rangle_B|dx=0$;\\
\item [(2)] $\lim\limits_{a\rightarrow\infty}\sup\limits_{B\subset\mathbb{R}^n,r_B=a}\frac{1}{\nu(B)}\int_{B}
|b(x)-\langle b\rangle_B|dx=0$;\\
\item [(3)] $\lim\limits_{a\rightarrow\infty}\sup\limits_{B\subset\mathbb{R}^n\backslash B(x_0,a)}\frac{1}{\nu(B)}\int_{B}|b(x)-\langle b\rangle_B|dx=0$,
\end{itemize}
where $x_0$ is arbitrarily fixed point in $\mathbb{R}^n$.
\end{definition}
The weighted $\rm{VMO}$ spaces are very different from the classical $\rm{VMO}$ spaces, in \cite{LL}, the authors constructed an example to show that weighted $\rm{VMO}$ spaces is not necessarily the closure of $C_c^\infty(\mathbb{R}^n)$ under the weighted $\rm{BMO}$ norm when the dimension $n\geq2$.

Let $p,\lambda_1,\lambda_2,\nu$ be stated as Definition \ref{weighted VMO} and $\lambda_i'=\lambda_i^{-1/(p-1)}$, $i=1,2$. Recently, Chen et al. \cite{ChLLV} introduced the following two new definitions for weighted $\rm{VMO}$ spaces by $\rm{VMO}_{\lambda_1\lambda_2}(\mathbb{R}^n)$, $\rm{VMO}_{\lambda_1'\lambda_2'}(\mathbb{R}^n)$:
\begin{definition}\label{equivalent1}
Let $1<p<\infty$ and $\lambda_1,\lambda_2\in A_p$, $\nu=\lambda_1^{1/p}\lambda_2^{-1/p}$. A function $b\in\rm{BMO}_\nu(\mathbb{R}^n)$ belongs to $\rm{VMO}_{\lambda_1\lambda_2}(\mathbb{R}^n)$ if
\begin{itemize}
\item [(1)] $\lim\limits_{a\rightarrow0}\sup\limits_{B\subset\mathbb{R}^n,r_B=a}\Big(\frac{1}{\lambda_1(B)}\int_{B}
|b(x)-\langle b\rangle_B|^{p}\lambda_2(x)dx\Big)^{1/p}=0$;\\
\item [(2)] $\lim\limits_{a\rightarrow\infty}\sup\limits_{B\subset\mathbb{R}^n,r_B=a}\Big(\frac{1}{\lambda_1(B)}\int_{B}
|b(x)-\langle b\rangle_B|^{p}\lambda_2(x)dx\Big)^{1/p}=0$;\\
\item [(3)] $\lim\limits_{a\rightarrow\infty}\sup\limits_{B\subset\mathbb{R}^n\backslash B(x_0,a)}\Big(\frac{1}{\lambda_1(B)}\int_{B}
|b(x)-\langle b\rangle_B|^{p}\lambda_2(x)dx\Big)^{1/p}=0$,
\end{itemize}
where $x_0$ is arbitrarily fixed point in $\mathbb{R}^n$.
\end{definition}
\begin{definition}\label{equivalent2}
Let $1<p<\infty$ and $\lambda_1,\lambda_2\in A_p$, $\nu=\lambda_1^{1/p}\lambda_2^{-1/p}$. A function $b\in\rm{BMO}_\nu(\mathbb{R}^n)$ belongs to $\rm{VMO}_{\lambda_1\lambda_2}(\mathbb{R}^n)$ if
\begin{itemize}
\item [(1)] $\lim\limits_{a\rightarrow0}\sup\limits_{B\subset\mathbb{R}^n,r_B=a}\Big(\frac{1}{\lambda_2'(B)}\int_{B}
|b(x)-\langle b\rangle_B|^{p'}\lambda_1'(x)dx\Big)^{1/p'}=0$;
\\
\item [(2)] $\lim\limits_{a\rightarrow\infty}\sup\limits_{B\subset\mathbb{R}^n,r_B=a}\Big(\frac{1}{\lambda_2'(B)}\int_{B}
|b(x)-\langle b\rangle_B|^{p'}\lambda_1'(x)dx\Big)^{1/p'}=0$;\\
\item [(3)] $\lim\limits_{a\rightarrow\infty}\sup\limits_{B\subset\mathbb{R}^n\backslash B(x_0,a)}\Big(\frac{1}{\lambda_2'(B)}\int_{B}
|b(x)-\langle b\rangle_B|^{p'}\lambda_1'(x)dx\Big)^{1/p'}=0$,
\end{itemize}
where $x_0$ is arbitrarily fixed point in $\mathbb{R}^n$.
\end{definition}
Moreover, Chen et al. \cite{ChLLV} proved that the spaces $\rm{VMO}_{\lambda_1\lambda_2}(\mathbb{R}^n)$, $\rm{VMO}_{\lambda_1'\lambda_2'}(\mathbb{R}^n)$ are equivalent to $\rm{VMO}_\nu(\mathbb{R}^n)$, where $\lambda_1$, $\lambda_2$, $\nu$, $\lambda_1'$ and $\lambda_2'$ are stated as in Definition \ref{equivalent2}.

\section{Two weight compactness of commutators}
In this section, we give the proofs of $(1)$ in Theorems \ref{theorem1.1}, \ref{theorem1.2}, \ref{theorem1.6}. Now, we are in the position to prove $(1)$ in Theorem \ref{theorem1.1}.
\begin{proof}[Proof of Theorem \ref{theorem1.1}]
(1).\quad Let $(\mathcal{T}_{\mathcal{S},b}^\alpha)^\ast$ be defined as \eqref{sparse3}. Recall from \cite{AMPRR}, we know that $(\mathcal{T}_{\mathcal{S},b}^\alpha)^\ast$ is bounded from $L^{p}(\lambda_1^p)$ to $L^{q}(\lambda_2^q)$. Thus, $\mathcal{T}_{\mathcal{S},b}^\alpha$ is compact if and only if $(\mathcal{T}_{\mathcal{S},b}^\alpha)^\ast$ is compact from $L^{p}(\lambda_1^p)$ to $L^{q}(\lambda_2^q)$. So it suffices to prove that $(\mathcal{T}_{\mathcal{S},b}^\alpha)^\ast$ is compact from $L^{p}(\lambda_1^p)$ to $L^{q}(\lambda_2^q)$.

For any $\epsilon>0$, we decompose
$$(\mathcal{T}_{\mathcal{S},b}^\alpha)^\ast f(x)=\mathcal{T}_{\epsilon,N_\epsilon}^\alpha f(x)+\mathcal{T}_\epsilon^\alpha f(x).$$
In the following, we will prove that for any $\epsilon>0$, there exists $N_\epsilon>0$ such that $\mathcal{T}_{\epsilon,N_\epsilon}^\alpha f(x)$ is a sparse operator with finite range, that is,
\begin{align*}
\mathcal{T}_{\epsilon,N_\epsilon}^\alpha f(x)=\sum_{k=1}^{N_\epsilon}a_k\chi_{Q_k}(x)
\end{align*}
and
\begin{align*}
\|\mathcal{T}_{\epsilon}^\alpha f\|_{L^q(\lambda_2^q)}\leq\epsilon\|f\|_{L^p(\lambda_1^p)}.
\end{align*}

According to Definitions \ref{weighted VMO}-\ref{equivalent2}, $\lambda_1^p,\lambda_2^p\in A_p$ and
$$\rm{VMO}(\mathbb{R}^n)=\rm{VMO}_{\lambda_1^p\lambda_2^p}(\mathbb{R}^n)
=\rm{VMO}_{\lambda_1^{-p'}\lambda_2^{-p'}}(\mathbb{R}^n),$$
one can select $N,\delta>0$ and cube $Q_N$ with side length $N$ such that
\begin{equation}\label{1}
\frac{1}{\nu(Q)}\int_Q|b(x)-\langle b\rangle_Q|dx<\epsilon,
\end{equation}
\begin{equation}\label{2}
\Big(\frac{1}{\lambda_1^p(Q)}\int_Q|b(x)-\langle b\rangle_Q|^p\lambda_2(x)^pdx\Big)^{1/p}<\epsilon
\end{equation}
and
\begin{equation}\label{3}
\Big(\frac{1}{\lambda_2^{-p'}(Q)}\int_Q|b(x)-\langle b\rangle_Q|^{p'}\lambda_1(x)^{-p'}dx\Big)^{1/p'}<\epsilon,
\end{equation}
provided that $l_Q>N$, $l_Q<\delta$ and $Q\cap Q_N=\emptyset$.

For $Q\in\mathcal{S}$, observe that
\begin{align*}
(\mathcal{T}_{\mathcal{S},b}^\alpha)^\ast(|f|)(x)&=\sum_{Q\supset Q_N}\Big(\frac{1}{|Q|^{1-\alpha/n}}\int_Q|b(y)-\langle b\rangle_Q||f(y)|dy\Big)\chi_Q(x)\\
&\quad+\sum_{Q\cap Q_N=\emptyset}\Big(\frac{1}{|Q|^{1-\alpha/n}}\int_Q|b(y)-\langle b\rangle_Q||f(y)|dy\Big)\chi_Q(x)\\
&\quad+\sum_{Q\subset Q_N,l_Q<\delta}\Big(\frac{1}{|Q|^{1-\alpha/n}}\int_Q|b(y)-\langle b\rangle_Q||f(y)|dy\Big)\chi_Q(x)\\
&\quad+\sum_{Q\subset Q_N,l_Q>\delta}\Big(\frac{1}{|Q|^{1-\alpha/n}}\int_Q|b(y)-\langle b\rangle_Q||f(y)|dy\Big)\chi_Q(x)\\
&=:I(x)+II(x)+III(x)+IV(x).
\end{align*}

Note that there are finitely many cubes contained in $Q_N$ satisfying $\delta<l_Q<N$, we have that $IV(x)$ is a sparse operator with finite range, and then $IV(x)$ is a compact operator.

It remains to show that
$$\|I\|_{L^q(\lambda_2^q)},\|II\|_{L^q(\lambda_2^q)},\|III\|_{L^q(\lambda_2^q)}\lesssim
\epsilon\|f\|_{L^p(\lambda_1^p)}.$$

Invoking Lemma \ref{lemma2.3}, for a.e. $y\in Q$, we have
\begin{equation}\label{4}
|b(y)-\langle b\rangle_Q|\lesssim\sum_{R\in\tilde{\mathcal{S}},R\subset Q}\frac{1}{|R|}\int_R|b(x)-\langle b\rangle_R|dx\chi_R(y).
\end{equation}
From this and \eqref{1}, we deduce that
\begin{align}\label{C}
III(x)&\lesssim\sum_{Q\subset Q_N,l_Q<\delta}\sum_{R\in\tilde{\mathcal{S}},R\subset Q}
\Big(\frac{1}{|R|}\int_R|b(z)-\langle b\rangle_R|dz\frac{1}{|Q|^{1-\alpha/n}}\int_R|f(y)|dy\Big)\chi_Q(x)\\
&=\sum_{Q\subset Q_N,l_Q<\delta}\sum_{R\in\tilde{\mathcal{S}},R\subset Q}
\frac{1}{\nu(R)}\int_R|b(z)-\langle b\rangle_R|dz\Big(\frac{1}{|R|}\int_R|f(y)|\nu(R)dy\Big)
\frac{\chi_Q(x)}{|Q|^{1-\alpha/n}}\nonumber\\
&\leq\epsilon\sum_{Q\subset Q_N,l_Q<\delta}\Big(\sum_{R\in\tilde{\mathcal{S}},R\subset Q}\langle|f|\rangle_R\nu(R)\Big)\frac{\chi_Q(x)}{|Q|^{1-\alpha/n}}\nonumber\\
&\leq\epsilon\sum_{Q\subset Q_N,l_Q<\delta}\frac{1}{|Q|^{1-\alpha/n}}
\Big(\int_Q \mathcal{T}_{\tilde{\mathcal{S}}}(|f|)(y)\nu(y)dy\Big)\chi_Q(x)\nonumber\\
&\leq\epsilon \mathcal{T}_{\mathcal{S}}^\alpha(\mathcal{T}_{\tilde{\mathcal{S}}}(|f|)\nu)(x),\nonumber
\end{align}
where $\mathcal{T}_{\mathcal{S}}$ and $\mathcal{T}_{\mathcal{S}}^\alpha$ are defined as \eqref{2.2} and \eqref{2.3}, respectively.
Let $p,q,\alpha$ be given as Theorem \ref{theorem1.1}. Since $\lambda_2\in A_{p,q}$ and $\lambda_1^p\in A_p$, by \eqref{sparse boundedness1} and \eqref{sparse boundedness2}, we have
\begin{align}\label{C norm}
\|III\|_{L^q(\lambda_2^q)}&\lesssim\epsilon\|\mathcal{T}_\mathcal{S}^\alpha
(\mathcal{T}_{\tilde{\mathcal{S}}}(|f|)\nu)\|_{L^q(\lambda_2^q)}\\
&\lesssim\epsilon[\lambda_2]_{A_{p,q}}^{(1-\alpha/n)\max\{1,p/q'\}}
\|\mathcal{T}_{\tilde{\mathcal{S}}}(|f|)\nu\|_{L^p(\lambda_2^p)}\nonumber\\
&=\epsilon[\lambda_2]_{A_{p,q}}^{(1-\alpha/n)\max\{1,p/q'\}}
\Big(\int_{\mathbb{R}^n}\mathcal{T}_{\tilde{\mathcal{S}}}(|f|)(x)^p\nu(x)^p\lambda_2(x)^pdx\Big)^{1/p}\nonumber\\
&=\epsilon[\lambda_2]_{A_{p,q}}^{(1-\alpha/n)\max\{1,p/q'\}}
\Big(\int_{\mathbb{R}^n}\mathcal{T}_{\tilde{\mathcal{S}}}(|f|)(x)^p\lambda_1(x)^pdx\Big)^{1/p}\nonumber\\
&\lesssim\epsilon[\lambda_2]_{A_{p,q}}^{(1-\alpha/n)\max\{1,p/q'\}}
[\lambda_1^p]_{A_p}^{\max\{1,\frac{1}{p-1}\}}\|f\|_{L^p(\lambda_1^p)}.\nonumber
\end{align}

While for $II$. Since $Q\cap Q_N=\emptyset$ and $R\subset Q$, we have that for all $R\in\mathcal{S}$ in \eqref{4} and $R\cap Q_N=\emptyset$. Then using \eqref{1} and following a similar argument as $III$, we also get that
\begin{align}\label{B}
\|II\|_{L^q(\lambda_2^q)}\lesssim\epsilon\|f\|_{L^p(\lambda_1^p)}.
\end{align}

Finally, we deal with $I$. Let us begin with a collection of sparse dyadic cubes $Q=Q_1\supset Q_2\supset Q_3\supset\cdots\supset Q_{\tau_Q}\supset Q_{\tau_Q+1}=Q_N$, where $Q_{i+1}$ is the child of $Q_{i+1}$ with $i=1,2,\cdots,\tau_Q$. To guarantee the sparse property, we still denote the parent of $Q_{i+1}$ by $Q_{i+1}$ provided that the parent of $Q_{i+1}$ has only one child $Q_{i+1}$. Repeat this process until we find $Q_i$ has at least two children with $Q_{i+1}$ being one of them. For every $Q_i$, we denote all its dyadic children except $Q_{i+1}$ by $Q_{i,k}$, $k=1,2,\cdots,M_{Q_i}$. Then the number of the children of $Q_i$ is $M_{Q_i}+1$, and we can find a uniform constant $M$ such that $M_{Q_i}+1\leq M$. Therefore, as in \cite{ChLLV}, we have that for any $i=1,2,\cdots,\tau_Q$ and $k=1,2,\cdots,M_{Q_i}$,
$$Q_{i,k}\cap Q_N=\emptyset,\quad |Q_{i+1}|\sim|Q_{i,k}|$$
and there is a uniform constant $\tilde{\eta}\in(0,1)$ such that $|Q_{i+1}|\leq\tilde{\eta}|Q_i|$, which guarantee the sparse property of $\{Q_i\}_i$. Hence,
\begin{align}\label{A}
I(x)&\leq\sum_{Q\supset Q_N}\Big(\sum_{i=1}^{\tau_Q}\sum_{k=1}^{M_{Q_i}}\frac{1}{|Q|^{1-\alpha/n}}\int_{Q_{i,k}}
|b(y)-\langle b\rangle_{Q_{i,k}}||f(y)|dy\Big)\chi_Q(x)\\
&\quad+\sum_{Q\supset Q_N}\Big(\frac{1}{|Q|^{1-\alpha/n}}\int_{Q_{N}}
|b(y)-\langle b\rangle_{Q_{N}}||f(y)|dy\Big)\chi_Q(x)\nonumber\\
&\quad+\sum_{Q\supset Q_N}\Big(\sum_{i=1}^{\tau_Q}\sum_{k=1}^{M_{Q_i}}|\langle b\rangle_Q-\langle b\rangle_{Q_{i,k}}|\frac{1}{|Q|^{1-\alpha/n}}
\int_{Q_{i,k}}|f(y)|dy\Big)\chi_Q(x)\nonumber\\
&\quad+\sum_{Q\supset Q_N}|\langle b\rangle_{Q_N}-\langle b\rangle_Q|\frac{1}{|Q|^{1-\alpha/n}}\int_{Q_{N}}|f(y)|dy\chi_Q(x)\nonumber\\
&=:I_1(x)+I_2(x)+I_3(x)+I_4(x).\nonumber
\end{align}

First, we consider $I_2(x)$. For $g\in L^{q'}(\lambda_2^{-q'})$ with $\|g\|_{L^{q'}(\lambda_2^{-q'})}\leq1$, by duality and H\"{o}lder's inequality, we have
\begin{align*}
\|I_2\|_{L^q(\lambda_2^q)}&=\sup_{\|g\|_{L^{q'}(\lambda_2^{-q'})}\leq1}\Big|\Big\langle
\sum_{Q\supset Q_N}\Big(\frac{1}{|Q|^{1-\alpha/n}}\int_{Q_N}|b(y)-\langle b\rangle_{Q_N}||f(y)|dy\Big)\chi_Q(x),g(x)
\Big\rangle\Big|\\
&\leq\sum_{Q\supset Q_N}\frac{1}{|Q|^{1-\alpha/n}}\int_{Q_N}|b(y)-\langle b\rangle_{Q_N}||f(y)|dy\int_{Q}|g(x)|dx\\
&\leq\sum_{Q\supset Q_N}\frac{1}{|Q|^{1-\alpha/n}}\Big(\int_{Q_N}|b(y)-\langle b\rangle_{Q_N}|^{p'}\lambda_1(y)^{-p'}dy\Big)^{1/p'}
\Big(\int_{Q_N}|f(y)^p\lambda_1(y)^pdx\Big)^{1/p}\\
&\quad\times\Big(\int_Q|g(x)|^{q'}\lambda_2(x)^{-q'}dx\Big)^{1/q'}\lambda_2^q(Q)^{1/q}\\
&\leq\sum_{Q\supset Q_N}\frac{1}{|Q|^{1-\alpha/n}}\Big(\frac{1}{\lambda_2^{-p'}(Q_N)}
\int_{Q_N}|b(y)-\langle b\rangle_{Q_N}|^{p'}\lambda_1(y)^{-p'}dy\Big)^{1/p'}\\
&\quad\times\|f\|_{L^p(\lambda_1^p)}\lambda_2^{-p'}(Q_N)^{1/p'}\lambda_2^q(Q)^{1/q}.
\end{align*}
From the definition of $A_{p,q}$, we have
\begin{align}\label{Apq}
\frac{\lambda_2^{-p'}(Q_N)^{1/p'}\lambda_2^q(Q_N)^{1/q}}{|Q_N|^{1-\alpha/n}}\leq[\lambda_2]_{A_{p,q}}^{1/q}.
\end{align}
Since $\lambda_2^q\in A_{q(n-\alpha)/n}$, there is a constant $\sigma>0$ such that $\lambda_2^q\in A_{q(n-\alpha)/n-\sigma}$. Then
\begin{align}\label{doubling}
\frac{\lambda_2^q(Q)}{\lambda_2^q(Q_N)}\leq C\Big(\frac{|Q|}{|Q_N|}\Big)^{q(n-\alpha)/n-\sigma}.
\end{align}
Recall that $Q\in\mathcal{S}$, it follows that
\begin{align}\label{sparse property}
\sum_{Q\supset Q_N,Q\in\mathcal{S}}\Big(\frac{|Q_N|}{|Q|}\Big)^{\sigma/q}\leq C.
\end{align}
Therefore, by \eqref{3} and \eqref{Apq}-\eqref{sparse property}, we deduce that
\begin{align}\label{A2}
&\|I_2\|_{L^q(\lambda_2^q)}\\
&\quad\leq\epsilon\|f\|_{L^p(\lambda_1^p)}\sum_{Q\supset Q_N}\frac{\lambda_2^{-p'}(Q_N)^{1/p'}\lambda_2^q(Q_N)^{1/q}}{|Q_N|^{1-\alpha/n}}
\Big(\frac{|Q_N|}{|Q|}\Big)^{1-\alpha/n}\frac{\lambda_2^q(Q)^{1/q}}{\lambda_2^q(Q_N)^{1/q}}\nonumber\\
&\quad\lesssim\epsilon\|f\|_{L^p(\lambda_1^p)}[\lambda_2]_{A_{p,q}}^{1/q}\sum_{Q\supset Q_N}\Big(\frac{|Q_N|}{|Q|}\Big)^{1-\alpha/n}\Big(\frac{|Q|}{|Q_N|}\Big)^{(n-\alpha)/n-\sigma/q}\nonumber\\
&\quad=\epsilon\|f\|_{L^p(\lambda_1^p)}[\lambda_2]_{A_{p,q}}^{1/q}\sum_{Q\supset Q_N}\Big(\frac{|Q_N|}{|Q|}\Big)^{\sigma/q}\lesssim\epsilon\|f\|_{L^p(\lambda_1^p)}.\nonumber
\end{align}

For $I_1$. We change the order of the summation for $Q$ and $k$, so we may assume that $Q_{i,k}=\emptyset$ provided that $M_{Q_i}\leq k\leq M$. Then
\begin{align*}
I_1(x)=\sum_{k=1}^M\sum_{Q\supset Q_N}\Big(\sum_{i=1}^{\tau_Q}\frac{1}{|Q|^{1-\alpha/n}}\int_{Q_{i,k}}
|b(y)-\langle b\rangle_{Q_{i,k}}||f(y)|dy\Big)\chi_Q(x).
\end{align*}
Since $Q_{i,k}\cap Q_N=\emptyset$ and $R\subset Q_{i,k}$, Lemma \ref{lemma2.3} shows that
\begin{align*}
|b(y)-\langle b\rangle_{Q_{i,k}}|\lesssim\sum_{R\in\tilde{\mathcal{S}},R\subset Q_{i,k}}\frac{1}{|R|}\int_R|b(z)-\langle b\rangle_R|dz.
\end{align*}
And using $R\cap Q_N=\emptyset$ and \eqref{1}, we also have
\begin{align*}
\frac{1}{\nu(R)}\int_R|b(z)-\langle b\rangle_R|dz<\epsilon.
\end{align*}
Hence, similar to the estimate of \eqref{C} and \eqref{C norm}, we get that
\begin{align}\label{A1}
\|I_1\|_{L^q(\lambda_2^q)}\lesssim\epsilon\|f\|_{L^p(\lambda_1^p)}.
\end{align}

In the end, we deal with $I_3,I_4$. We first claim that for each fixed $k$ and $Q_{i,k}$, $i=1,2,\cdots,\tau_Q$.
$$\|A_{Q_{i,k}}\|_{L^q(\lambda_2^q)}\lesssim\epsilon\|f\|_{L^p(\lambda_1^p)},$$
where the implicit constant is independent of the cube $Q_{i,k}$ and
$$A_{Q_{i,k}}:=\sum_{Q\supset Q_N}\Big(\sum_{i=1}^{\tau_Q}|\langle b\rangle_{Q_{i,k}}-\langle b\rangle_Q|\frac{1}{|Q|^{1-\alpha/n}}\int_{Q_{i,k}}|f(y)|dy
\Big)\chi_Q(x).$$
Indeed, by $Q_j\supset Q_N$ and \eqref{1}, we have
\begin{align*}
|\langle b\rangle_{Q_{i,k}}-\langle b\rangle_Q|&\leq|\langle b\rangle_{Q_{i,k}}-\langle b\rangle_{Q_{i-1}}|+|\langle b\rangle_{Q_{i-1}}-\langle b\rangle_{Q_{i-2}}|+\cdots+|\langle b\rangle_{Q_2}-\langle b\rangle_Q|\\
&\leq\frac{1}{|Q_{i,k}|}\int_{Q_{i,k}}|b(x)-\langle b\rangle_{Q_{i-1}}|dx+\cdots+\frac{1}{|Q_2|}
\int_{Q_2}|b(x)-\langle b\rangle_Q|dx\\
&\lesssim\sum_{j=1}^{i-1}\frac{\nu(Q_j)}{|Q_j|}\Big(\frac{1}{\nu(Q_j)}\int_{Q_j}|b(x)-\langle b\rangle_{Q_j}|dx\Big)\leq\epsilon\sum_{j=1}^{i-1}\frac{\nu(Q_j)}{|Q_j|}.
\end{align*}
From this, we obtain
\begin{align*}
A_{Q_{i,k}}\lesssim\epsilon\sum_{Q\supset Q_N}\sum_{i=1}^{\tau_Q}\sum_{j=1}^{i-1}\frac{\nu(Q_j)}{|Q_j|}\frac{1}{|Q|^{1-\alpha/n}}
\int_{Q_{i,k}}|f(y)|dy\chi_Q(x).
\end{align*}
Therefore, by duality and H\"{o}lder's inequality, we have
\begin{align}\label{claim}
&\|A_{Q_{i,k}}\|_{L^q(\lambda_2^q)}\\
&\quad\lesssim\epsilon\sup_{\|g\|_{L^{q'}(\lambda_2^{-q'})\leq1}}
\Big|\Big\langle\sum_{Q\supset Q_N}\sum_{i=1}^{\tau_Q}\sum_{j=1}^{i-1}\frac{\nu(Q_j)}{|Q_j|}\frac{1}{|Q|^{1-\alpha/n}}
\int_{Q_{i,k}}|f(y)|dy\chi_Q(x),g(x)\Big\rangle\Big|\nonumber\\
&\quad\leq\epsilon\Big[\sum_{Q\supset Q_N}\sum_{i=1}^{\tau_Q}\sum_{j=1}^{i-1}\Big(\frac{\nu(Q_j)}{|Q_j|}\Big)^q
\Big(\int_{Q_{i,k}}|f(y)|dy\Big)^q\frac{1}{|Q|^{q-\alpha q/n}}\Big(\frac{|Q|}{|Q_{i,k}|}\Big)^{q\sigma'}\lambda_2^q(Q)\Big]^{1/q}\nonumber\\
&\qquad\times\Big[\sum_{Q\supset Q_N}\sum_{i=1}^{\tau_Q}\sum_{j=1}^{i-1}\Big(\frac{|Q_{i,k}|}{|Q|}\Big)^{q'\sigma'}
\Big(\int_Q|g(y)|dy\Big)^{q'}(\lambda_2^q(Q))^{-q'}\lambda_2^q(Q)\Big]^{1/q'},\nonumber
\end{align}
where $\sigma'=\frac{\sigma}{2q}$.
Note that $\{Q_i\}$ is a sparse family, there exists a constant $C>0$ such that
$$\sum_{i=1}^{\tau_Q}\log\Big(\frac{|Q|}{|Q_i|}\Big)\Big(\frac{|Q_i|}{|Q|}\Big)^{q'\sigma'}\leq C.$$
By \eqref{Ap property}, we have
\begin{align}\label{part1}
&\sum_{Q\supset Q_N}\sum_{i=1}^{\tau_Q}\sum_{j=1}^{i-1}\Big(\frac{|Q_{i,k}|}{|Q|}\Big)^{q'\sigma'}
\Big(\int_Q|g(y)|dy\Big)^{q'}(\lambda_2^q(Q))^{-q'}\lambda_2^q(Q)\\
&\quad\leq\sum_{Q\supset Q_N}\Big[\sum_{i=1}^{\tau_Q}\log\frac{|Q|}{|Q_i|}\Big(\frac{|Q_i|}{|Q|}\Big)^{q'\sigma'}\Big]
\lambda_2^q(Q)\nonumber\\
&\qquad\times\Big(\frac{1}{\lambda_2^q(Q)}\int_Q|g(y)|\lambda_2(y)^{-q}\lambda_2(y)^{q}dy
\Big)^{q'}\nonumber\\
&\quad\lesssim\sum_{Q\supset Q_N}\inf_{x\in Q}M_{\lambda_2^q}(|g|\lambda_2^{-q})(x)^{q'}\lambda_2^q(E_Q)\nonumber\\
&\quad\leq\sum_{Q\supset Q_N}\int_{E_Q}M_{\lambda_2^q}(|g|\lambda_2^{-q})(x)^{q'}\lambda_2(x)^qdx\nonumber\\
&\quad\leq\int_{\mathbb{R}^n}M_{\lambda_2^q}(|g|\lambda_2^{-q})(x)^{q'}\lambda_2(x)^qdx\nonumber\\
&\quad\lesssim\|g\lambda_2^{-q}\|_{L^{q'}(\lambda_2^q)}^{q'}=\|g\|_{L^{q'}(\lambda_2^{-q'})}^{q'}
\leq1.\nonumber
\end{align}

On the other hand, since
\begin{align*}
\nu(Q_j)^q=\Big(\int_{Q_j}\lambda_1\lambda_2^{-1}\Big)^q\leq\Big(\int_{Q_j}\lambda_1^q\Big)
\Big(\int_{Q_j}\lambda_2^{-q'}\Big)^{q/q'}=\lambda_1^q(Q_j)\lambda_2^{-q'}(Q_j)^{q/q'},
\end{align*}
we have
\begin{align}\label{part2}
&\sum_{Q\supset Q_N}\sum_{i=1}^{\tau_Q}\sum_{j=1}^{i-1}\Big(\frac{\nu(Q_j)}{|Q_j|}\Big)^q
\Big(\int_{Q_{i,k}}|f(y)|dy\Big)^q\frac{1}{|Q|^{q-\alpha q/n}}\Big(\frac{|Q|}{|Q_{i,k}|}\Big)^{q\sigma'}\lambda_2^q(Q)\\
&\quad\leq\sum_{Q\supset Q_N}\sum_{i=1}^{\tau_Q}\sum_{j=1}^{i-1}\|f\|_{L_{\lambda_1^p}^q(Q_{i,k})}^q
\frac{\lambda_1^{q}(Q_j)\lambda_2^{-q'}(Q_j)^{q/q'}}{|Q_j|^q}
\frac{\lambda_1^{-p'}(Q_{i,k})^{q/p'}\lambda_2^q(Q)}{|Q|^{q-\alpha q/n}}(|Q|/|Q_{i,k}|)^{q\sigma'}\nonumber\\
&\quad\leq\sum_{Q\supset Q_N}\sum_{i=1}^{\tau_Q}\sum_{j=1}^{i-1}\|f\|_{L_{\lambda_1^p}^p(Q_{i,k})}^q
(|Q_i|/|Q|)^{q-\alpha q/n}\frac{\lambda_1^{-p'}(Q_i)^{q/p'}\lambda_1^q(Q_i)}{|Q_i|^{q-q\alpha/n}}
\frac{\lambda_1^q(Q_j)}{\lambda_1^q(Q_i)}\nonumber\\
&\qquad\times\frac{\lambda_2^{-q'}(Q_j)^{q/q'}\lambda_2^q(Q_j)}{|Q_j|^q\lambda_2^q(Q_j)}
(|Q|/|Q_{i,k}|)^{q\sigma'}\lambda_2^q(Q)\nonumber\\
&\quad\lesssim[\lambda_1]_{A_{p,q}}[\lambda_2^q]_{A_q}\sum_{Q\supset Q_N}\sum_{i=1}^{\tau_Q}\sum_{j=1}^{i-1}\|f\|_{L_{\lambda_1^p}^p(Q_{i,k})}^q
(|Q|/|Q_j|)^{q-q\alpha/n-\sigma}\nonumber\\
&\qquad\times(|Q_i|/|Q|)^{q-\alpha q/n}(|Q_j|/|Q_i|)^{q-q\alpha/n-\sigma}
(|Q|/|Q_{i,k}|)^{q\sigma'}\nonumber\\
&\quad\leq[\lambda_1]_{A_{p,q}}[\lambda_2^q]_{A_q}\sum_{i=1}^\infty
\sum_{Q\supset Q_i}\log\Big(\frac{|Q|}{|Q_i|}\Big)(|Q_i|/|Q|)^{\sigma-q\sigma'}
\|f\|_{L_{\lambda_1^p}^p(Q_{i,k})}^q\nonumber\\
&\quad\lesssim[\lambda_1]_{A_{p,q}}[\lambda_2^q]_{A_q}\sum_{i=1}^\infty
\|f\|_{L_{\lambda_1^p}^p(Q_{i,k})}^q\nonumber\\
&\quad\leq[\lambda_1]_{A_{p,q}}[\lambda_2^q]_{A_q}
\|f\|_{L^p(\lambda_1^p)}^q,\nonumber
\end{align}
where in the third inequality, we use the fact that $\lambda_1^q,\lambda _2^q\in A_{q(n-\alpha)/n-\sigma}$ for some $\sigma>0$ and
\begin{align*}
\frac{\lambda_1^q(Q_j)}{\lambda_1^q(Q_i)}\lesssim
\Big(\frac{|Q_j|}{|Q_i|}\Big)^{\frac{q(n-\alpha)}{n}-\sigma},~
\frac{\lambda_2^q(Q)}{\lambda_2^q(Q_j)}\lesssim
\Big(\frac{|Q|}{|Q_j|}\Big)^{\frac{q(n-\alpha)}{n}-\sigma}.
\end{align*}

From \eqref{claim}-\eqref{part2}, we prove the claim. Since the implicit constant is independent of the cube $Q_{i,k}$, by \eqref{A}, we have
\begin{align*}
\|I_3\|_{L^q(\lambda_2^q)},\|I_4\|_{L^q(\lambda_2^q)}\lesssim\epsilon\|f\|_{L^p(\lambda_1^p)}.
\end{align*}
This together with \eqref{A2} and \eqref{A1} shows that
\begin{align}
\|I\|_{L^q(\lambda_2^q)}\lesssim\epsilon\|f\|_{L^p(\lambda_1^p)}.
\end{align}
Then, combining with \eqref{C norm} and \eqref{B}, we prove the result.
\end{proof}

Next, we prove $(1)$ of Theorem \ref{theorem1.2}, Corollary \ref{corollary1.3} and Theorem \ref{theorem1.6}.
\begin{proof}[Proof of Theorem \ref{theorem1.2}]
(1).\quad For fixed $x\in\mathbb{R}^n$ and $r>0$, note that
\begin{align}\label{lower estimate}
&\int_{\mathbb{R}^n}\frac{|b(x)-b(y)||f(y)|}{|x-y|^{n-\alpha}}dy\\
&\quad\geq\int_{|x-y|\leq r}\frac{|b(x)-b(y)||f(y)|}{|x-y|^{n-\alpha}}dy\nonumber\\
&\quad\geq\frac{1}{r^{n-\alpha}}\int_{|x-y|\leq r}|b(x)-b(y)||f(y)|dy,\nonumber
\end{align}
Taking supremum for $r>0$ on both sides of \eqref{lower estimate}, we obtain
$$\int_{\mathbb{R}^n}\frac{|b(x)-b(y)||f(y)|}{|x-y|^{n-\alpha}}dy\geq M_\alpha^bf(x).$$
By carefully checking the proof in \cite[Theorem 2.1]{AMPRR}, one can check that
$$\int_{\mathbb{R}^n}\frac{|b(x)-b(y)||f(y)|}{|x-y|^{n-\alpha}}dy\lesssim
\sum_{j=1}^{3^n}\mathcal{T}_{\mathcal{S}_j,b}^\alpha(|f|)(x)+
(\mathcal{T}_{\mathcal{S}_j,b}^\alpha)^\ast(|f|)(x).$$
Hence, $(1)$ of Theorem \ref{theorem1.2} follows from $(1)$ of Theorem \ref{theorem1.1}.
\end{proof}

\begin{proof}[Proof of Corollary \ref{corollary1.3}]
(1).\quad Note that $b$ is non-negative, then
\begin{align*}
|[b,M_\alpha]f(x)|&=|M_\alpha(bf)(x)-b(x)M_\alpha f(x)|\\
&=|M_\alpha(bf)(x)-M_\alpha(b(x)f)(x)|\\
&\leq M_\alpha(bf-b(x)f)(x)=M_\alpha^bf(x).
\end{align*}
Thus, the conclusion follows by $(1)$ in Theorem \ref{theorem1.2}.
\end{proof}

\begin{proof}[Proof of Theorem \ref{theorem1.6}]
(1).\quad It was proved in \cite{AMPRR} that
$$|[b,I_\alpha]f(x)|\lesssim\sum_{j=1}^{3^n}\mathcal{T}_{\mathcal{S}_j,b}^\alpha(|f|)(x)+
(\mathcal{T}_{\mathcal{S}_j,b}^\alpha)^\ast(|f|)(x).$$
Hence, the conclusion follows by $(1)$ in Theorem \ref{theorem1.1}.
\end{proof}

\section{proofs of (2) of Theorems \ref{theorem1.1}, \ref{theorem1.2}, \ref{theorem1.6}}
In this section, we prove (2) of Theorems \ref{theorem1.1}, \ref{theorem1.2}, \ref{theorem1.6}. Note that if the following order hold:
\begin{align*}
b\in{\rm VMO}_{\nu}(\mathbb{R}^n)\Rightarrow\forall\mathcal{S},\mathcal{T}_{\mathcal{S},b}^\alpha ~compact\Rightarrow M_b^\alpha~ compact\Rightarrow b\in\rm{VMO}_{\nu}(\mathbb{R}^n),
\end{align*}
\begin{align*}
b\in{\rm VMO}_{\nu}(\mathbb{R}^n)\Rightarrow\forall\mathcal{S},\mathcal{T}_{\mathcal{S},b}^\alpha ~compact\Rightarrow [b,I_\alpha]~ compact\Rightarrow b\in\rm{VMO}_{\nu}(\mathbb{R}^n),
\end{align*}
we complete the proofs of $(2)$ in Theorems \ref{theorem1.1}, \ref{theorem1.2}, \ref{theorem1.6}. Therefore, it suffices to show $(2)$ in Theorems \ref{theorem1.2}, \ref{theorem1.6}.
We first recall the following relevant definition.
\begin{definition} {\rm (cf. \cite{LOR1})}
By a median value of a real-valued measurable function $f$ over a measure set $E$ of positive finite measure, we mean a possibly non-unique, real number $m_f(E)$ such that
$$\max(|\{x\in E: f(x)>m_f(E)\}|,\,\,|\{x\in E: f(x)<m_f(E)\}|)\leq|E|/2.$$
\end{definition}

\begin{lemma}\label{new lemma}
Let $\alpha\in(0,n)$, $1<p<n/\alpha$, $1/p-1/q=\alpha/n$. Assume that $\lambda_1,\lambda_2\in A_{p,q}$, $\nu=\lambda_1/\lambda_2$, then for any cube $Q\subset \mathbb{R}^n$, we have
\begin{align*}
\frac{1}{\nu(Q)}\lesssim\frac{|Q|^{\alpha/n}}{\lambda_1^p(Q)^{1/p}\lambda_2^{-q'}(Q)^{1/q'}}.
\end{align*}
\end{lemma}
\begin{proof}
Since $\lambda_1,\lambda_2\in A_{p,q}$, it is easy to see that $\nu\in A_2$. Using $p<q$, the H\"{o}lder inequality yields that
\begin{align*}
\Big(\frac{1}{|Q|}\int_Q\lambda_2(x)^{p}dx\Big)^{1/p}\leq
\Big(\frac{1}{|Q|}\int_Q\lambda_2(x)^{q}dx\Big)^{1/q}.
\end{align*}
Then using $\nu\in A_2$, $\lambda_1^p\in A_p,\lambda_2^q\in A_q$, $1/p-1/q=\alpha/n$ and the H\"{o}lder inequality, it follows that
\begin{align*}
\frac{1}{\nu(Q)}\lesssim\frac{\nu^{-1}(Q)}{|Q|^2}&=\frac{1}{|Q|^2}\int_Q\lambda_1(x)^{-1}
\lambda_2(x)dx\\
&\leq\frac{1}{|Q|^2}\Big(\int_Q\lambda_1(x)^{-p'}dx\Big)^{1/p'}
\Big(\int_Q\lambda_2(x)^{p}dx\Big)^{1/p}\\
&=\frac{1}{|Q|}\Big(\frac{1}{|Q|}\int_Q\lambda_1(x)^{-p'}dx\Big)^{1/p'}
\Big(\frac{1}{|Q|}\int_Q\lambda_2(x)^{p}dx\Big)^{1/p}\\
&\leq\frac{1}{|Q|}\frac{\Big(\frac{1}{|Q|}\int_Q\lambda_1(x)^{-p'}dx\Big)^{1/p'}
\Big(\frac{1}{|Q|}\int_Q\lambda_1(x)^{p}dx\Big)^{1/p}}
{\Big(\frac{1}{|Q|}\int_Q\lambda_1(x)^{p}dx\Big)^{1/p}}\\
&\quad\times\frac{\Big(\frac{1}{|Q|}\int_Q\lambda_2(x)^{q}dx\Big)^{1/q}
\Big(\frac{1}{|Q|}\int_Q\lambda_2(x)^{-q'}dx\Big)^{1/q'}}
{\Big(\frac{1}{|Q|}\int_Q\lambda_2(x)^{-q'}dx\Big)^{1/q'}}\\
&\lesssim\frac{|Q|^{\alpha/n}}{\lambda_1^p(Q)^{1/p}\lambda_2^{-q'}(Q)^{1/q'}}.
\end{align*}
\end{proof}

\begin{proof}[Proof of Theorem \ref{theorem1.2}]
(2).\quad Suppose that $b\in\rm{BMO_\nu(\mathbb{R}^n)}$ such that $M_\alpha^b$ is compact from $L^p(\lambda_1^p)$ to $L^q(\lambda_2^q)$. We will prove the result via a contradiction argument, so we assume that $b\not\in\rm{VMO_\nu(\mathbb{R}^n)}$.

The approach is as follows: on any Hilbert space $\mathcal{H}$, with canonical basis $e_j$, $j\in\mathbb{N}$, an operator with $Te_j=v$, with non-zero $v\in\mathcal{H}$, is necessarily unbounded. For $b\in\rm{BMO_\nu(\mathbb{R}^n)}\backslash\rm{VMO_\nu(\mathbb{R}^n)}$, we will prove a variant of this condition for the commutator $M_\alpha^b$ giving us the contradiction.

By the assumption of $b\not\in\rm{VMO_\nu(\mathbb{R}^n)}$, we see that $b$ fails to satisfy at least one of the three conditions in Definition \ref{weighted VMO}. We just show $(1)$ of Definition \ref{weighted VMO} does not hold since other cases are similar. According to Definition \ref{weighted VMO}, there exist $\epsilon_0>0$ and a sequence $\{B_j\}_{j=1}^\infty:=\{B_j(x_j,r_j)\}_{j=1}^\infty$ of balls such that $r_j\rightarrow0$ as $j\rightarrow\infty$ and that
\begin{align}\label{fanmian}
\frac{1}{\nu(B_j)}\int_{B_j}|b(x)-\langle b\rangle_{B_j}|dx\geq\epsilon_0.
\end{align}
Since $r_j\rightarrow0$ as $j\rightarrow\infty$, we can further assume without loss of generality that
\begin{align}\label{decay}
4r_{j+1}\leq r_j.
\end{align}
And for each $B_j$, one can select another ball $\tilde{B}_j$ such that $\tilde{r}_j=r_j$, $B_j\cap\tilde{B}_j=\emptyset$ and $d(B_j,\tilde{B}_j)=\inf_{x\in B_j,y\in\tilde{B}_j}|x-y|\leq5\tilde{r}_j$, where $\tilde{r}_j$ is the radius of ball $\tilde{B_j}$.

For $j=1,2,\cdots$, we denote the median value of $b$ on the ball $\tilde{B}_j$ by $m_b(\tilde{B}_j)$, set
$$F_{j,1}\subset\{y\in\tilde{B_j}:b(y)\leq m_b(\tilde{B}_j)\},\quad F_{j,2}\subset\{y\in\tilde{B_j}:b(y)\geq m_b(\tilde{B}_j)\}$$
such that
\begin{align}\label{Fj1}
|F_{j,1}|,|F_{j,2}|\geq\frac{1}{2}|\tilde{B}_j|.
\end{align}
We also set
$$E_{j,1}=\{x\in B_j:b(x)\geq m_b(\tilde{B}_j)\},\quad E_{j,2}=\{x\in B_j:b(x)< m_b(\tilde{B}_j)\}.$$
Then it is direct that $E_{j,1}\cup E_{j,2}=B_j$, $E_{j,1}\cap E_{j,2}=\emptyset$,
\begin{align*}
b(x)-b(y)\geq0\quad for~ (x,y)\in E_{j,1}\times F_{j,1},
\end{align*}
\begin{align*}
b(x)-b(y)<0\quad for~ (x,y)\in E_{j,2}\times F_{j,2},
\end{align*}
and for $(x,y)\in(E_{j,1}\times F_{j,1})\cup(E_{j,2}\times F_{j,2})$,
\begin{align}\label{invarient}
|b(x)-b(y)|=|b(x)-m_b(\tilde{B}_j)|+|m_b(\tilde{B}_j)-b(y)|\geq|b(x)-m_b(\tilde{B}_j)|.
\end{align}

Define
$$\tilde{F}_{j,1}:=F_{j,1}\backslash\cup_{l=j+1}^\infty\tilde{B_l},\quad \tilde{F}_{j,1}:=F_{j,2}\backslash\cup_{l=j+1}^\infty\tilde{B_l}.$$
Applying \eqref{decay} and \eqref{Fj1}, we have
\begin{align}\label{4.5}
|\tilde{F}_{j,1}|&\geq|F_{j,1}|-|\cup_{l=j+1}^\infty\tilde{B}_l|\\
&\geq\frac{1}{2}|\tilde{B}_j|-\sum_{l=j+1}^\infty|\tilde{B_l}|\nonumber\\
&\geq\frac{1}{2}|\tilde{B}_j|-\frac{1}{3}|\tilde{B}_j|=\frac{1}{6}|\tilde{B}_j|.\nonumber
\end{align}
Similarly, we also get that $|\tilde{F}_{j,2}|\geq\frac{1}{6}|\tilde{B}_j|$.

Note that
\begin{align*}
&\frac{1}{\nu(B_j)}\int_{B_j}|b(x)-\langle b\rangle_{B_j}|dx\\
&\quad\leq\frac{2}{\nu(B_j)}
\int_{B_j}|b(x)-m_b(\tilde{B}_j)|dx\\
&\quad=\frac{2}{\nu(B_j)}\Big(\int_{E_{j,1}}|b(x)-m_b(\tilde{B}_j)|dx+
\int_{E_{j,2}}|b(x)-m_b(\tilde{B}_j)|dx\Big).
\end{align*}
From this and \eqref{fanmian}, we have that at least one of the following inequalities holds:
\begin{align*}
\frac{2}{\nu(B_j)}\int_{E_{j,1}}|b(x)-m_b(\tilde{B}_j)|dx\geq\frac{\epsilon_0}{2},~
\frac{2}{\nu(B_j)}\int_{E_{j,2}}|b(x)-m_b(\tilde{B}_j)|dx\geq\frac{\epsilon_0}{2}.
\end{align*}
In the following, we only consider the first case since the other is similar. Invoking \eqref{invarient}, \eqref{4.5} and Lemma \ref{new lemma}, we deduce that
\begin{align*}
\frac{\epsilon_0}{4}&\leq\frac{1}{\nu(B_j)}\int_{E_{j,1}}|b(x)-m_b(\tilde{B}_j)|dx\\
&\lesssim\frac{|\tilde{F}_{j,1}|}{\nu(B_j)|B_j|}\int_{E_{j,1}}|b(x)-m_b(\tilde{B}_j)|dx\\
&\lesssim\frac{1}{\lambda_1^p(B_j)^{1/p}\lambda_2^{-q'}(B_j)^{1/q'}}\int_{E_{j,1}}
\int_{\tilde{F}_{j,1}}\frac{1}{|B_j|^{1-\alpha/n}}|b(x)-m_b(\tilde{B}_j)|dydx\\
&\leq\frac{1}{\lambda_2^{-q'}(B_j)^{1/q'}}\int_{E_{j,1}}\frac{1}{|\tilde{B}_j|^{1-\alpha/n}}
\int_{\tilde{B}_j}|b(x)-b(y)|f_j(y)dydx\\
&\leq\frac{1}{\lambda_2^{-q'}(B_j)^{1/q'}}\int_{E_{j,1}}M_\alpha^b(f_j)(x)dx,
\end{align*}
where $f_j:=\frac{\chi_{\tilde{F}_{j,1}}}{\lambda_1^p(B_j)^{1/p}}$. One can see that
$$\|f_j\|_{L^p(\lambda_1^p)}\sim1,$$
and $\{f_j\}$ is a sequence of disjointly supported functions. Furthermore, by H\"{o}lder's inequality
\begin{align*}
\epsilon_0&\lesssim\frac{1}{\lambda_2^{-q'}(B_j)^{1/q'}}\int_{E_{j,1}}M_\alpha^b(f_j)(x)
\lambda_2(x)\lambda_2(x)^{-1}dx\\
&\leq\frac{1}{\lambda_2^{-q'}(B_j)^{1/q'}}\big(\lambda_2^{-q'}(E_{j,1})\big)^{1/q'}
\Big(\int_{\mathbb{R}^n}M_\alpha^b(f_j)(x)^q\lambda_2(x)^qdx\Big)^{1/q}\\
&\leq\Big(\int_{\mathbb{R}^n}M_\alpha^b(f_j)(x)^q\lambda_2(x)^qdx\Big)^{1/q}.
\end{align*}

Now, we return to the assumption of compactness. Let $\psi$ be in the closure of $\{M_\alpha^b(f_j)\}_j$, then
$$\|\psi\|_{L^q(\lambda_2^q)}\gtrsim1.$$
And we select $j_i$ such that
$$\|\psi-M_\alpha^b(f_{j_i})\|_{L^q(\lambda_2^q)}\leq2^{-i}.$$
Let us consider a non-negative numerical sequence $\{c_i\}_i$ with $\|\{c_i\}\|_{l^{q'}}<\infty$ but $\|\{c_i\}\|_{l^{1}}=\infty$. Then for $\phi=\sum_ic_if_{j_i}\in L^p(\lambda_1^p)$, we have
\begin{align*}
&\Big\|\sum_i c_i\psi-M_\alpha^b\phi\Big\|_{L^q(\lambda_2^q)}\\
&\quad\leq\Big\|\sum_i c_i(\psi-M_\alpha^b(f_{j_i}))\Big\|_{L^q(\lambda_2^q)}\\
&\quad
\leq\|\{c_i\}\|_{l^{q'}}\Big[\sum_i\|\psi-M_\alpha^b(f_{j_i})\|_{L^q(\lambda_2^q)}^q
\Big]^{1/q}\lesssim1,
\end{align*}
which implies that $\sum_ic_i\psi\in L^q(\lambda_2^q)$. However, $\sum_ic_i\psi$ is infinite on a set of positive measure, this is a contradiction. Thus, we complete the proof.
\end{proof}

To show $(2)$ of Theorem \ref{theorem1.6}, we establish the following lemma.
\begin{lemma}\label{new lemma1}
Let $K_\alpha(x,y)=\frac{1}{|x-y|^{n-\alpha}}$. Then for each $A\geq4$ and each ball $B:=B(x_0,r)$, there exists a disjoint ball $\tilde{B}:=B(y_0,r)$ with dist$(B,\tilde{B})\sim Ar$, such that
$$K_\alpha(x,y)\gtrsim\frac{1}{r^{n-\alpha}}.$$
\end{lemma}
\begin{proof}
Fix a ball $B:=B(x_0,r)$. For each $A\geq4$, take $y_0=x_0+Ar\theta_0$, where $\theta_0\in\mathbb{S}^{n-1}$. Let $\tilde{B}:=B(y_0,r)$, it is easy to see that dist$(B,\tilde{B})\sim Ar$ and for any $(x,y)\in(B\times\tilde{B})$,
$$|x-y|\leq(A+2)r,$$
It follows that
$$K_\alpha(x,y)=\frac{1}{|x-y|^{n-\alpha}}\gtrsim\frac{1}{r^{n-\alpha}}.$$
\end{proof}
\begin{proof}[Proof of Theorem \ref{theorem1.6}]
(2).\quad Suppose that $b\in\rm{BMO_\nu(\mathbb{R}^n)}$ such that $[b,I_\alpha]$ is compact from $L^p(\lambda_1^p)$ to $L^q(\lambda_2^q)$. We will prove the result via a contradiction argument, so we assume that $b\not\in\rm{VMO_\nu(\mathbb{R}^n)}$.

The approach is as follows: on any Hilbert space $\mathcal{H}$, with canonical basis $e_j$, $j\in\mathbb{N}$, an operator with $Te_j=v$, with non-zero $v\in\mathcal{H}$, is necessarily unbounded. For $b\in\rm{BMO_\nu(\mathbb{R}^n)}\backslash\rm{VMO_\nu(\mathbb{R}^n)}$, we will prove that a variant of this condition for the commutator $[b,I_\alpha]$ giving us the contradiction.

By the assumption of $b\not\in\rm{VMO_\nu(\mathbb{R}^n)}$, we see that $b$ fails to satisfy at least one of the three conditions in Definition \ref{weighted VMO}. We still show $(1)$ of Definition \ref{weighted VMO} does not hold. According to Definition \ref{weighted VMO}, there exist $\epsilon_0>0$ and a sequence $\{B_j\}_{j=1}^\infty:=\{B_j(x_j,r_j)\}_{j=1}^\infty$ of balls such that $r_j\rightarrow0$ as $j\rightarrow\infty$ and \eqref{fanmian}, \eqref{decay} hold.

Then, as in the proof of $(2)$ in Theorem \ref{theorem1.2}, we can define the sets $E_{j,1}$, $E_{j,2}$, $\tilde{F}_{j,1}$ and $\tilde{F}_{j,2}$. As before, we also consider the following case:
\begin{align*}
\frac{2}{\nu(B_j)}\int_{E_{j,1}}|b(x)-m_b(\tilde{B}_j)|dx\geq\frac{\epsilon_0}{2}.
\end{align*}
Applying \eqref{invarient} and Lemmas \ref{new lemma}, \ref{new lemma1}, we deduce that
\begin{align*}
\frac{\epsilon_0}{4}&\leq\frac{1}{\nu(B_j)}\int_{E_{j,1}}|b(x)-m_b(\tilde{B}_j)|dx\\
&\lesssim\frac{|\tilde{F}_{j,1}|}{\nu(B_j)|B_j|}\int_{E_{j,1}}|b(x)-m_b(\tilde{B}_j)|dx\\
&\lesssim\frac{1}{\lambda_1^p(B_j)^{1/p}\lambda_2^{-q'}(B_j)^{1/q'}}\int_{E_{j,1}}
\int_{\tilde{F}_{j,1}}K_\alpha(x,y)|b(x)-m_b(\tilde{B}_j)|dydx\\
&\leq\frac{1}{\lambda_2^{-q'}(B_j)^{1/q'}}\Big|\int_{E_{j,1}}
\int_{\tilde{B}_j}K_\alpha(x,y)(b(x)-b(y))f_j(y)dydx\Big|\\
&\leq\frac{1}{\lambda_2^{-q'}(B_j)^{1/q'}}\int_{E_{j,1}}|[b,I_\alpha](f_j)(x)|dx,
\end{align*}
where $f_j:=\frac{\chi_{\tilde{F}_{j,1}}}{\lambda_1^p(B_j)^{1/p}}$. One can see that
$$\|f_j\|_{L^p(\lambda_1^p)}\sim1,$$
and $\{f_j\}$ is a sequence of disjointly supported functions. Furthermore, by H\"{o}lder's inequality
\begin{align*}
\epsilon_0&\lesssim\frac{1}{\lambda_2^{-q'}(B_j)^{1/q'}}\int_{E_{j,1}}[b,I_\alpha]b(f_j)(x)
\lambda_2(x)\lambda_2(x)^{-1}dx\\
&\leq\frac{1}{\lambda_2^{-q'}(B_j)^{1/q'}}\big(\lambda_2^{-q'}(E_{j,1})\big)^{1/q'}
\Big(\int_{\mathbb{R}^n}[b,I_\alpha](f_j)(x)^q\lambda_2(x)^qdx\Big)^{1/q}\\
&\leq\Big(\int_{\mathbb{R}^n}[b,I_\alpha](f_j)(x)^q\lambda_2(x)^qdx\Big)^{1/q}.
\end{align*}

Now, we return to the assumption of compactness. Let $\psi$ be in the closure of $\{[b,I_\alpha](f_j)\}_j$, then
$$\|\psi\|_{L^q(\lambda_2^q)}\gtrsim1.$$
And we select $j_i$ such that
$$\|\psi-[b,I_\alpha](f_{j_i})\|_{L^q(\lambda_2^q)}\leq2^{-i}.$$
Let us consider a non-negative numerical sequence $\{c_i\}_i$ with $\|\{c_i\}\|_{l^{q'}}<\infty$ but $\|\{c_i\}\|_{l^{1}}=\infty$. Then for $\phi=\sum_ic_if_{j_i}\in L^p(\lambda_1^p)$, we have
\begin{align*}
&\Big\|\sum_i c_i\psi-[b,I_\alpha]\phi\Big\|_{L^q(\lambda_2^q)}\\
&\quad\leq\Big\|\sum_i c_i(\psi-[b,I_\alpha](f_{j_i}))\Big\|_{L^q(\lambda_2^q)}\\
&\quad
\leq\|\{c_i\}\|_{l^{q'}}\Big[\sum_i\|\psi-[b,I_\alpha](f_{j_i})\|_{L^q(\lambda_2^q)}^q
\Big]^{1/q}\lesssim1,
\end{align*}
which implies that $\sum_ic_i\psi\in L^q(\lambda_2^q)$. However, $\sum_ic_i\psi$ is infinite on a set of positive measure, this is a contradiction. Thus, we complete the proof.
\end{proof}


\begin{thebibliography}{99}


\bibitem{AMPRR}N. Accomazzo, J.C. Mart\'{i}nez-Perales and I.P. Rivera-R\'{i}os,
On Bloom type estimates for iterated commutators of fractional integrals,
Indiana Univ. Math. J. 69(4) (2020), 1207--1230.

\bibitem{BeDMT}A. B\'{e}nyi, W. Dami\'{a}n, K. Moen and R.H. Torres,
Compact bilinear commutators: the weighted case,
Michigan Math. J. 64(1) (2015), 39--51.

\bibitem{Bl}S. Bloom,
A commutator theorem and weighted BMO,
Trans. Amer. Math. Soc. 292(1) (1985), 103--122.

\bibitem{ChaT}L. Chaffee and R.H. Torres,
Characterization of compactness of the commutators of bilinear fractional integral
operators,
Potential Anal. 43(3) (2015), 481--494.

\bibitem{ChD}Y. Chen and Y. Ding,
Compactness characterization of commutators for Littlewood-Paley operators,
Kodai Math. J. 32(2) (2009), 256--323.

\bibitem{ChDW}Y. Chen, Y. Ding and X. Wang,
Compactness of commutators for singular integrals on Morrey spaces,
Canad. J. Math. 64(2) (2012), 257--281.

\bibitem{ChLLV}P. Chen, M. Lacey, J. Li and M.N. Vempati,
Compactness of the Bloom sparse operators and applications,
arXiv: 2204.11990v1.

\bibitem{ClC}A. Clop and V. Cruz,
Weighted estimates for Beltrami equations,
Anal. Acad. Sci. Fenn. Math. 38(1) (2013), 91--113.

\bibitem{CRW}R.R. Coifman, R. Rochberg and G. Weiss,
Factorization theorems for Hardy spaces in several variables,
Ann. of Math. 103 (1976), 611-635.

\bibitem{GHWY}W. Guo, J. He, H. Wu and D. Yang,
Boundedness and compactness of commutators associated with Lipschitz functions,
Anal. Appl. 20(1) (2022), 35--71.

\bibitem{GWY}W. Guo, H. Wu and D. Yang,
A revisit on the compactness of commutators,
Canad. J. Math. 73(6) (2021), 1667--1697.

\bibitem{GZ}W. Guo and G. Zhao,
On relatively compact sets in quasi-Banach function spaces,
Proc. Amer. Math. Soc. 148(8) (2020), 3359--3373.

\bibitem{HLW}I. Holmes, M.T. Lacey and B.D. Wick,
Commutators in the two-weight setting,
Math. Ann. 367(1-2) (2016), 51--80.

\bibitem{HRS}I. Holmes, R. Rahm and S. Spencer,
Commutators with fractional operatros,
Studia Math. 233(3) (2016), 279--291.

\bibitem{Jan}S. Janson,
Mean oscillation and commutators of singular integral operators,
Ark. Mat. 16(2) (1978), 263--270.

\bibitem{LL}M. Lacey and J. Li, Compactness of commutator of Riesz transforms in the two
    weight setting, J. Math. Anal. Appl. 508(1) (2022), 11pp.

\bibitem{LMPT}M. Lacey, K. Moen, C. P\'{e}rez and R.H. Torres,
Sharp weighted bounds for fractional integral operators,
J. Funct. Anal. 259(5) (2010), 1073--1097.

\bibitem{Ler}A.K. Lerner,
On pointwise estimates involving sparse operators,
New York J. Math. 22 (2017), 341--349.

\bibitem{LOR}A.K. Lerner, S. Ombrosi and I.P. Rivera-R\'{i}os,
On pointwise and weighted estimates for commutators of Calder\'{o}n-Zygmund operators,
Adv. Math. 319 (2017), 153--181.

\bibitem{LOR1}A.K. Lerner, S. Ombrosi and I.P. Rivera-R\'{i}os,
Commutators of singular integrals revisited,
Bull. Lond. Math. Soc. 51 (2019), 107--119.

\bibitem{LiWY}S. Liu, H. Wu and D. Yang,
A note on extrapolation of compactness,
Collect. Math. Doi:10.1007/s13348-022-00353-w.

\bibitem{ST}C. Segovia and J.L. Torrea,
Weighted inequalities for commutators of fractional and singular integrals,
Publ. Mat. 35(1) (1991), 209--235.

\bibitem{ShFL}S. Shi, Z. Fu and S. Lu,
On the compactness of commutators of Hardy operators,
Pacific J. Math. 307(1) (2020), 239--256.

\bibitem{TX}R.H. Torres and Q. Xue,
On compactness of commutators of multiplication and bilinear pseudodifferential operators and a new subspace of \rm{BMO},
Rev. Mat. Iberoam. 36(3) (2020), 939--956.

\bibitem{U}A. Uchiyama,
On the compactness of operators of Hankel type,
Tohoku Math. J.(2) 30(1) (1978), 163--171.

\bibitem{WaX}S. Xue and Q. Xue,
On weighted compactness of commutators of bilinear maximal Calderón-Zygmund singular integral operators.
Forum Math. 34(2) (2022), 307--322.

\bibitem{Wa}S. Wang,
The compactness of the commutator of fractional integral operator,
Chin. Ann. Math (A). 8 (1987), 475--482.

\bibitem{WY}H. Wu and D. Yang,
Characterizations of weighted compactness of commutators via $\rm{CMO}(\mathbb{R}^n)$,
Proc. Amer. Math. Soc. 146(10) (2018), 4239--4254.

\bibitem{XYY}Q. Xue, K. Yabuta and J. Yan,
Weighted Fr\'{e}chet-Kolmogorov theorem and compactness of vector-valued multilinear operators. J. Geom. Anal. 31(10) (2021), 9891--9914.

\end{thebibliography}
\end{document}